\documentclass{article}

\usepackage{color} 
\usepackage{amssymb,amsmath,mathrsfs,amsthm}
\usepackage{enumitem}
\usepackage{natbib}
\usepackage{mathtools}
\usepackage{geometry}
\usepackage{amsfonts}

\usepackage{esint,comment}

\usepackage[dvips]{graphicx}
 \newtheorem{theorem}{Theorem}[section]
 \newtheorem{corollary}[theorem]{Corollary}
 \newtheorem{lemma}[theorem]{Lemma}
 
 \newtheorem{proposition}[theorem]{Proposition}
 \theoremstyle{definition}
 \newtheorem{definition}[theorem]{Definition}
 \theoremstyle{remark}
 \newtheorem{remark}[theorem]{Remark}
 
 \numberwithin{equation}{section}

\def \no#1#2#3 {{\bf #1} (#3), #2.}
\def \eds#1#2#3 {#1, #2, #3.}

\def\d{{\rm d}}

\def\:{{\colon}}

\def\be#1{\begin{equation}\label{#1}}
\def\ee{\end{equation}}

\def\<{\langle}
\def\>{\rangle}
\def\coloneqq{:=}

\newcommand{\p}{\partial}

\newcommand{\NN}{\mathbb{N}}
\newcommand{\RR}{\mathbb{R}}
\newcommand{\eqnb}{\begin{equation}}
\newcommand{\eqne}{\end{equation}}

\newcommand\blfootnote[1]{%
  \begingroup
  \renewcommand\thefootnote{}\footnote{#1}%
  \addtocounter{footnote}{-1}%
  \endgroup
}

\begin{document}
\title{On regularity properties of a surface growth model}
\author{Jan Burczak, Wojciech S. O\.za\'nski, Gregory Seregin}
\date{\vspace{-5ex}}
\maketitle
\blfootnote{J.~Burczak: Institute of Mathematics, Polish Academy of Sciences,   \'Sniadeckich 8, 00-656 Warsaw, Poland and Mathematisches Institut, Universit\"at Leipzig, Augustusplatz 10, D-04109 Leipzig, Germany, email: burczak@math.uni-leipzig.de  \\ 
W. S. O\.za\'nski: Department of Mathematics, University of Southern California, Los Angeles, CA 90089, USA, email: ozanski@usc.edu\\
 G.~Seregin: Oxford University, UK and St Petersburg Department of Steklov Mathematical Institute, RAS, Russia, email: seregin@maths.ox.ac.uk\\
 J. Burczak was supported by MNiSW "Mobilno\'s\'c Plus" grant 1289/MOB/IV/2015/0. \\
 W. S. O\.za\'nski was supported by postdoctoral funding from ERC 616797, the AMS Simons Travel Grant and by funding from Charles Simonyi Endowment at the Institute for Advanced Study.\\
G. Seregin was supported by the grant RFBR 20-01-00397. }

\begin{abstract}
We show local higher integrability of derivative of a suitable weak solution to the surface growth model, provided a scale-invariant quantity is locally bounded. If additionally our scale-invariant quantity is small, we prove local smoothness of solutions.
\end{abstract}


\section{Introduction}
We consider the one-dimensional scalar surface growth model (SGM) 
\begin{equation}\label{SGMe}
v_t +v_{xxxx} = - \p_{xx}(v_x)^2,
\end{equation}
which is a model of epitaxial growth of monocrystals, with $v$ being the height of a crystalline layer. For more applicational motivations see \cite{Si-Pl:94b}, \cite{Ra-Li-Ha:00a}, \cite{raible_mayr_linz_moske_hanggi_samwer}, \cite{hoppe_nash} and \cite{hoppe_nash}, and for certain stochastic aspects see \cite{blomker_raible}, \cite{blomker_hairer}, \cite{blomker_gugg}, \cite{blomkerflandoliromito09} and \cite{blomker_romito_stochastic}. The analytical results for the SGM share striking similarities with the 3D incompressible Navier--Stokes equations, which has been explored in the recent years by \cite{SteWin:05}, Bl\"omker, Romito (2009, 2012)\nocite{blomkerromito09,blomkerromito12}, \cite{SGM} and \cite{O_Serrin_cond}.

The Cauchy problem associated with \eqref{SGMe} 
\begin{equation}\label{SGMC}
\begin{cases}
v_t +v_{xxxx} = - \p_{xx}(v_x)^2\qquad \text{ in } \RR \times (0,\infty ) =: Q_+\\
v(0)=v_0 \qquad \qquad \qquad \qquad  \text{ on } \RR,
\end{cases}
\end{equation}
admits the following notion of an energy weak solution
\begin{definition} 
Let $ v_0\in L_2 (\RR)$. We say that a function $v \in L_\infty(0, \infty; {L}_2 (\RR))$ such that $v_{xx}\in L_2 (0, \infty; L_2 (\RR) )$ is an {\emph{energy weak solution}} to \eqref{SGMC} provided
	\begin{itemize}
\item[(i.]  distributional formulation)
\eqnb\label{weak_form_R}
\int_0^\infty \int_\RR \left( v \phi_t - v_{xx} \phi_{xx} - v_x^2 \phi_{xx} \right) = \int_\RR v_0 \phi (0)
\eqne
for every $\phi \in C_0^\infty ((-1,\infty)\times \RR)$,	
\item[(ii.]  energy inequality) for almost every $t \ge 0$
\eqnb\label{EI_on_R} \frac 12 \int_\RR  v(t)^2+\int\limits^t_0\int_\RR v_{xx}^2 \leq \frac 12 \int_\RR v_0^2.\eqne
	\end{itemize}
\end{definition}
\begin{remark}\label{rem:str}
Any energy weak solution $v$ of \eqref{SGMC} can be modified on a set of measure zero so that the energy inequality holds for any $t \ge 0$ and 
\begin{itemize}
\item[(iii)] 	$v(t)$ is weakly continuous into $L_2$, i.e.\ for any $w\in L_2$
\[
\int_\RR v(t) w \quad \text{ is continuous on } [0, \infty),
\]
\item[(iv)] (time-truncated distributional formulation) for every $t>0$ and $\phi \in C_0^\infty ((-1,\infty )\times \RR)$
\eqnb\label{weak_form_R_with_cutoff}
\int_\RR v(t) \phi(t) + \int_0^t \int_\RR \left( v \phi_t - v_{xx} \phi_{xx} - v_x^2 \phi_{xx} \right) = \int_\RR v_0 \phi (0),
\eqne
\item[(v)] $\| v(t) - v_0 \|_{L_2 (\RR )} \to 0$ as $t\to 0^+$.
	\end{itemize}
\end{remark}
Remark \ref{rem:str} is proven in Appendix \ref{sec_app_on_weak_sol}.
\vskip 2mm
In this note we are interested in local regularity properties of weak solutions to  \eqref{SGMe}. Throughout the paper we will use the notation $z=(x,t)$ and we will denote a biparabolic cylinder centred at $z$ by
\[
Q_r (z) \coloneqq (x-r,x+r)\times (t-r^4,t].
\]
We will use the following shorthand notation for \\
(a) cylinders:
\[
 Q_r (z) = Q_r \quad \text{ as well as } \quad Q_1 (0,0) \equiv Q,
\]
where there is no danger of confusing cylinders with different centres,
and for\\
(b) the function spaces:
\[
\begin{aligned}
L_{p,q}(Q_r) := L_q(t-r^4,t;{L}_p  (x-r,x+r) ), \quad &{W}{^{k,0}_{p,q}}(Q_r) := L_q(t-r^4,t;{W}^k_p  (x-r,x+r) ),\\
\end{aligned}
\]
We will also apply the convention that any Sobolev spaces and Lebesgue spaces are considered on $\RR$ unless specified otherwise. Moreover, we write $\| \cdot \|_q := \| \cdot \|_{L^q}$, $\int  := \int_\RR$. 
We are now ready to introduce
\begin{definition}\label{def:sws}
Function $v$ is a {\emph{suitable weak solution to \eqref{SGMe} on $Q$}}, provided
	\begin{itemize}
\item[(vi)] 	$v\in L_{2, \infty}(Q)$ and $v_{xx} \in L_{2,2} (Q)$,

\item[(vii)] 	$v$ satisfies \eqref{SGMe}  in the sense of distributions on $Q$, i.e. $\int \int \left(v(\phi_t - \phi_{xxxx}) -v_x^2 \phi_{xx}\right)=0$ for every $\phi \in C_0^\infty ((-1,1)\times (-1,0))$,	
\item[(viii)] (local energy inequality) For any nonnegative $ \phi \in C_0^\infty (Q)$ and almost any $t \in (-1, 0)$ 
\eqnb\label{LEI_alt_form}\begin{aligned}
\frac{1}{2} \int_{-1}^1 v^2(t) \phi(t) + &\int_{-1}^t \int_{-1}^1 |v_{xx}|^2\phi \leq \\
\int_{-1}^t& \int_{-1}^1 \left(\frac{1}{2}(\phi_t-\phi_{xxxx})v^2
+2|v_x|^2\phi_{xx}-\frac{5}{3}v_x^3\phi_x-|v_x|^2v\phi_{xx} \right).
\end{aligned}
\eqne
	\end{itemize}
\end{definition}
In view of the definition of $Q$, the test function $\phi \in C_0^\infty (Q)$ may not vanish at $t=0$, it vanishes only in the neighbourhood of the parabolic boundary. 
\begin{remark}\label{rem:str_loc}
As in (iv), we can modify a suitable weak solution $v$ on a set of measure zero such that the generalized distributional formulation
\eqnb\label{weak_form_cylinder_with_cutoff}
\int_{-1}^1 v(t) \phi(t) + \int_{-1}^t \int_{-1}^1 \left( v \phi_t - v_{xx} \phi_{xx} - |v_x|^2 \phi_{xx} \right) =  0
\eqne
holds for all $\phi \in C_0^\infty (Q) $ and $t\in (-1,0)$, and that the local energy inequality \eqref{LEI_alt_form} holds for every $t\in (-1,0)$.
\end{remark}
Definition \ref{def:sws} admits the obvious generalisation to an arbitrary cylinder $Q_r (z)$. 

The definitions of weak energy solution and of suitable weak solution are not artificial. Indeed, integration by parts yields the cancellation
\[
\int_{\RR } v\,\p_{xx} (v_x)^2  = - \int_{\RR }  (v_x)^2 v_{xx} =0,
\]
which enables the following existence result.
\begin{theorem}\label{thm:Cauchy}
Given $v_0\in L_2 (\RR)$ there exists an energy weak solution to the Cauchy problem \eqref{SGMC} that is a suitable weak solution on every cylinder  $Q_r (z)$.
\end{theorem}
For the sake of completeness, the proof of Theorem \ref{thm:Cauchy} is given in Appendix \ref{sec_app_on_weak_sol}.
\subsection{Supercriticality}
Our equation \eqref{SGMe} enjoys invariance under the scaling
\begin{equation}\label{sc}
v^\lambda (x,t)= v (\lambda x, \lambda^4 t)
\end{equation}
and its total energy
\[
E (v) \coloneqq \sup_{t \ge 0} \int_\RR |v(t)|^2 + 2\int_0^\infty \int_\RR |v_{xx} (s) |^2 \d s
\] 
vanishes on small scales, i.e.\ $E (v^\lambda)= \lambda^{-1} E (v) \to 0$ as $\lambda \to \infty$. In this the sense equation \eqref{SGMe} is supercritical, hence one expects that standard methods, e.g.\ a perturbation of linear theory, do not provide a satisfactory answer to well-posedness in the large of \eqref{SGMC}. Observe that the small-scale decay of $E$ occurs at the same rate $\lambda^{-1}$ as in the case of 3D Navier-Stokes Equations  under its scaling. This essentially leads to certain similarities between \eqref{SGMe} and 3D Navier-Stokes Equations, as discussed in the literature.

\subsection{Result}
We provide regularity statements in relation to the quantity
\begin{equation}\label{eq:sinv}
I (z,R) \coloneqq \frac{1}{R^2}  \int_{Q_R (z)} |v_x|^3.
\end{equation}
We note that $I(z,R)$ is invariant with respect to the scaling \eqref{sc} and it has been studied in the previous work \cite{SGM}, where smoothness of $v$ in $Q_{R/2}(z)$ has been deduced from smallness of $I(z,R)$.
In our main result, i.e.\ Theorem \ref{thm_main} below, we improve integrability of a suitable weak solution for which $I$ is bounded on small scales, and we show local smoothness if $I$ is small.

We will denote by $C_\alpha$ any positive constant that depends on some parameter $\alpha$.
 \begin{theorem}\label{thm_main}
Let $v$ be a suitable weak solution of the SGM \eqref{SGMe} on  $Q$. \\
(i) (higher integrability) If there exists $M <\infty$ such that
\begin{equation}\label{eq:bass}
\sup_{Q_R (z) \subset Q} \frac{1}{R^2}  \int_{Q_R (z)} |v_x|^3  \leq M,
\end{equation}
then there exists $\delta_0= \delta_0 (M) >0$ such that $v_x \in L_{\frac{10+\delta_0}{3} ,\frac{10+\delta_0}{3}} (Q_\frac{1}{2})$ with
\begin{equation}\label{eq:geh}
 \left( \int_{Q_\frac{1}{2}} | v_{x}|^\frac{10+\delta_0}{3}\right)^\frac{3}{10 + \delta_0} \le C_M  \left(\int_{Q} |v_{x}|^\frac{10}{3}\right)^\frac{3}{10} \quad \left(\le C C_M \left(\|v\|_{2, \infty; Q} + \|v_{xx}\|_{2,2; Q}  \right) \right)
\end{equation}
(ii) ($\varepsilon$-regularity) Given $\gamma \in (0,1)$ there exists $\varepsilon >0$ such that if 
\begin{equation}\label{eq:sass}
I (0,1)\equiv   \int_{Q_1 } |v_x|^3\leq \varepsilon,
\end{equation}
then 
\begin{equation}\label{eq:hcon}
\begin{aligned}
&\left| v_x (x,t) - v_x (y,s) \right| \leq C_\gamma \varepsilon^{1/3}  \left( |x-y| + |t-s |^{1/4} \right)^\gamma \quad \text{ for } (x,t),(y,s)\in Q_{\frac{1}{2}}
\end{aligned}
\end{equation}
\end{theorem}
Let us note that:
\begin{itemize}
\item The higher integrability result \eqref{eq:geh} does not follow from regularity of a suitable weak solution by interpolation etc.
\item Our result holds for an arbitrary $\gamma \in (0,1)$, unlike in the case of Navier-Stokes Equations (compare Theorem 3.1 in \cite{lin} and \cite{ladyzhenskaya_seregin}), thanks to lack of pressure-related difficulties.
\item \cite{SGM} showed that smallness of $I$ implies H\"older continuity of $v$ (rather than of $v_x$), which is not sufficient to further bootstrap the regularity of $v$ and exhibits a mismatch between an assumption involving $v_x$ and a result for $v$. Part (ii) fills this gap, as \eqref{eq:hcon} guarantees smoothness. Indeed, using for example the regularity condition of \cite{O_Serrin_cond}, one has
 \begin{corollary}
Under assumptions of Theorem \ref{thm_main} (ii),
$v \in C^\infty ( Q_{1/8} )$.
 \end{corollary}
\item One can also provide respective partial regularity results based on Theorem  \ref{thm_main}.
 In particular, denoting by 
\[\begin{split}
S &\coloneqq \left\{ (x,t) \in \RR \times (0,\infty ) \colon v\text{ is not infinitely differentiable }\right.\\
&\hspace{5cm}\left.\text{ on any neighbourhood } U \ni (x,t) \right\}
\end{split}
\]
the singular set of a suitable weak solution $v$, one can deduce from part (ii) that $\mathcal{P}^1 (S) =0$ and that $d_B (S \cap K) \leq 7/6 $ for every compact $K\subset \RR \times (0,\infty )$, where $\mathcal{P}^1$ denotes the one-dimensional biparabolic Hausdorff measure 
and $d_B$ stands for the box-counting dimension. This improves the conclusion of \cite{SGM}, who showed these estimates with $S$ replaced by the set of points where $v$ is not H\"older continuous (a subset of $S$).  
\end{itemize}

The remaining part of this note is devoted to proof of Theorem \ref{thm_main}.


\section{Auxiliary tools}
Here we gather two tools needed further: a Campanato-type estimate for a linear equation and a multiplicative inequality. 
\subsection{The linear equation}\label{ssec:prl}
 Consider 
\begin{equation}\label{eq:linb}
u_t + u_{xxxx} + \beta u_{xxx} =0,
\end{equation}
where $\beta \in \RR$ is a parameter. We will use \eqref{eq:linb} as a limiting system in our blowup-type proof of the $\varepsilon$-regularity result (i.e. Theorem \ref{thm_main} (ii)). In fact, one of the main purposes of this work is to demonstrate that \eqref{eq:linb} is a more optimal linearisation of the surface growth model \eqref{SGMe} than the biharmonic heat equation $u_t+u_{xxxx}=0$ (which was used by \cite{SGM}). In fact, if one of the $v_x$'s in $(v_x)^2$ in \eqref{SGMe} is replaced by a constant $\beta$, then the nonlinearity becomes $\beta \,v_{xxx}$, which motivates \eqref{eq:linb}. In this section we discuss a Campanato-type estimate for \eqref{eq:linb}, which will be later used in closing the blowup-type argument (in Section \ref{sec_blowup_limit_eq}).
\begin{lemma}[Campanato-type estimate for the linear equation]\label{lem_campanato_linear}
Suppose that $u\in L_{2,2} (Q_1)$ with $u_x \in L_{2,2} (Q_1)$ satisfies \eqref{eq:linb} in the sense of distributions. Then for any $\theta \in (0,1/2)$, $p\geq 1$,
\eqnb\label{campanato_linear}
\left( \fint_{Q_\theta} \left| u_x - (u_x)_\theta \right|^p \right)^{1/p} \leq c(1+\beta) \, \theta \|u_x\|_{L_{2,2}(Q_1)},
\eqne
where $c>1$ is a universal constant and $(u_x)_\theta\coloneqq \fint_{Q_\theta} u_x $.
\end{lemma}
\begin{proof} See Appendix \ref{sec_proof_linear_camp}.
\end{proof}
\begin{remark}
Observe that there is no lowest-order term $\|u\|_{L_{2,2}(Q_1)}$ present on the right-hand side of \eqref{campanato_linear}, which the simplest energy estimate would dictate. Instead \eqref{campanato_linear} follows from an introduction of time-dependent oscillations in estimates, inspired by \cite{GiaquintaStruwe82}, see also \cite{SereginSverakSilvestreZlatos}.
\end{remark}
\subsection{Multiplicative inequality}
In the proof of Theorem \ref{thm_main} (ii) (i.e. our $\varepsilon$-regularity result) it is sufficient to use standard interpolation. However for the proof of Theorem \ref{thm_main} (i) we need the following more precise inequality. Recall $Q_+\coloneqq \RR \times (0,\infty )$.
\begin{proposition}\label{prop_mhi} Let $U = Q_r (z)$ or $U = Q_+$. There exists a constant $C>0$ (independent of $U $) such that 
\begin{equation}\label{eq:mhi}
\|f_x\|_{L_{\frac{10}{3}, \frac{10}{3}} ( U)} \le C \|f\|^\frac{2}{5}_{L_{2, \infty}( U )}  \|f_{xx}\|^\frac{3}{5}_{L_{2,2} ( U)} 
\end{equation}
for every $f\in L_{2, \infty}(U)$ with $f_{xx}\in L_{2,2}(U)$ such that, in the case of $U=Q_r (z)$, $f$ is compactly supported in $(x-r, x+r)$. 
\end{proposition}
We note that the compact support requirement may be relaxed to the condition that $f$ and $f_x$ vanish at at least one point in space (at almost every $t$) or the condition that they vanish in the sense of spatial average. 
\begin{proof}[Proof of Proposition \ref{prop_mhi}]
From the Gabushin inequality 
\[
\|u_x\|_{\frac{10}{3}} \le C \|u\|^\frac{2}{5}_{2}  \|u_{xx}\|^\frac{3}{5}_{2},
\]
(cf.\ Theorem 1.4 of \cite{KwongZettl}, p.12, or the original \cite{Gab67}) and therefore also on an  interval $I$, provided $u$ is compactly supported there (it actually suffices that $u$ and $u_x$ vanish somewhere on $I$, cf \cite{KwongZettl}, Corollary 1.1, p. 21). Hence by density argument we have the same result at almost every\ $t$ for $u(t)$, where $u \in {W}{_{2,2}^{2,0}}(U)$. Taking both sides to power $10/3$ and integrating in time yields \eqref{eq:mhi}.
\end{proof} 

\section{Caccioppoli inequality}
Let $\varphi_0 \in C_0^\infty (-1,1)$ be a nonnegative and even cut-off function, such that $\varphi_0 =1$ on $(-1/2,1/2)$, and, given $x_0\in \RR$ and $R>0$ let 
\[\varphi_{x_0, R} (x) \coloneqq \varphi_0 ((x-x_0)/R). \]
We introduce the $\varphi_0$\emph{-related mean} of a function $f$
\begin{equation}\label{eq:mmv}
\langle f\rangle_{x_0,  R}  \coloneqq \int_{x_0-R}^{x_0+R} f \, \varphi_{x_0,  R}  \, \bigg(\,  \int_{x_0-R}^{x_0+R}  \varphi_{x_0,  R} \bigg)^{-1}.
\end{equation}
\begin{remark}\label{rem:ct}
If $v$ is a suitable weak solution on $Q_R$ then $\p_t \langle v\rangle_{x_0,R}\in L_{5/3} (t_0-R^4,t_0)$ (in particular $\langle v\rangle_{x_0,R}(t)$ is a continuous function of $t$). Indeed, abbreviating $\varphi\coloneqq \varphi_{x_0,R}$, we have for every $\psi \in C_0^\infty (t_0-R^4,t_0)$ via the distributional formulation (vii)
\eqnb\label{eq:mmv1}\begin{split}
\left| \int_{t_0-R^4}^{t_0} \langle v\rangle_{x_0,R} (t) \psi' (t) \d t \right| &= \frac{1}{\int_{x_0-R}^{x_0+R} \varphi } \left| \int_{Q_R} v \varphi \psi' \right| = C_\varphi \left| \int_{Q_R} v_{x}^2 \varphi_{xx} \psi - v_x  \varphi_{xxx} \psi \right| \\
\leq &C_\varphi \| v_x \|_{L_{10/3,10/3} (Q_R)}^2 \| \psi \|_{L_{5/2} (t_0-R^4,t_0)} \\
&+C_\varphi \| v_x \|_{L_{10/3,10/3} (Q_R)} \| \psi \|_{L_{10/7} (t_0-R^4,t_0)} \\
\leq &C_{\varphi,R,v }   \| \psi \|_{L_{5/2} (t_0-R^4,t_0)},
\end{split}
\eqne
 where we used \eqref{eq:mhi} and (vi) in the last line.
\end{remark}
We will also use a smooth nonnegative time cut-off function $\chi_0 \in C^\infty (\RR )$ such that $\chi_0 (t) \equiv 0 $ for $t\leq -1$ and $\chi_0 \equiv 1 $ for $t\geq -1/16$. Let 
\[
\chi_{t_0,R} (t) \coloneqq \chi_0 ((t-t_0)/R^4),
\]
then $\chi_{t_0, R}(t) \equiv 0 \;\text{ for }\; t \leq t_0-R^4$ and $\chi_{t_0, R}(t) \equiv 1 \;\text{ for }\; t \geq t_0-(R/2)^4$. We now set the space-time cutoff by writing
\eqnb\label{eq:def_eta}
\eta_{z_0,  R}(x,t) \coloneqq \chi_{t_0,  R}(t) \; \varphi_{x_0,  R}(x).
\eqne
Finally, given a function $f$, $R>0$ and $x_0\in \RR$ we will write 
\eqnb\label{eq:hatf}
\hat f (x,t)\coloneqq  f (x,t) - \langle f(t)\rangle_{x_0, R}.
\eqne

\begin{proposition}[Caccioppoli inequality]
Let $v$ be a suitable weak solution to \eqref{SGMe} on $Q_R (z_0) = ({t_0 - R^4}, {t_0}] \times (x_0-R, x_0+R)$, then 
\begin{equation}\label{eq:cacc}
\sup_{t \in (t_0 - R^4, t_0)}  \int_{x_0-R}^{x_0+R} |\hat v|^2 \eta_{z_0,  R} \d x+ \int_{Q_R (z_0)} |\hat v_{xx}|^2 \eta_{z_0,  R} \le \frac{C}{R^2}  \int_{Q_R (z_0)} |\hat v_{x}|^2 + \frac{C}{R}  \int_{Q_R (z_0)} |\hat v_{x}|^3.
\end{equation}
\end{proposition}
The main issue in the proof consists in replacing  $v$ with the time dependent oscillation $\hat v$ in the local energy inequality. 
\begin{proof}
Without loss of generality we can assume $z_0=0$ and $R=1$. The case of general $z_0$ and $R>0$ follows then from dilations and shifts, due to scale invariance of \eqref{eq:cacc}. We will write $\eta \coloneqq \eta_{0,1}$ for brevity. For brevity we will skip the variable under the integrals below; instead every integral ``$\int_{-1}^1$'' is taken with respect to $x$ and we will write argument ``$(t)$'' to point out that the integral is taken at a given time $t$. Letting $c(t) \coloneqq \langle v(t)\rangle_{0,1}$ we have $v = \hat{v} + c(t)$, and so the local energy inequality \eqref{LEI_alt_form} with $\phi \coloneqq \eta$ gives 
\[\begin{split}
&\frac{1}{2} \int_{-1}^1 (\hat{v} +c)^2(t) \eta (t) + \int_{-1}^t \int_{-1}^1 |\hat{v}_{xx}|^2\eta \\
&\leq \int_{-1}^t \int_{-1}^1 \left(\frac{1}{2}(\eta_t-\eta_{xxxx})  (\hat{v} +c)^2 +2 |\hat{v}_x|^2\eta_{xx}-\frac{5}{3}   \hat{v}_x^3\eta_x-  |\hat{v}_x|^2 (\hat{v} +c) \eta_{xx} \right)
\end{split}
\]
for every $t\in (-1,0)$ (recall Remark \ref{rem:str_loc}). Rearranging this inequality so that all terms involving $c$ are moved to the right-hand side yields
\begin{equation}\label{eq:pcac1}
\begin{aligned}
&\frac{1}{2} \int_{-1}^1 \hat{v}^2(t) \eta(t)  + \int_{-1}^t \int_{-1}^1 |\hat{v}_{xx}|^2\eta \\
&\hspace{3cm}- \int_{-1}^t \int_{-1}^1 \left(\frac{1}{2}(\eta_t-\eta_{xxxx})  \hat{v}^2 +2 |\hat{v}_x|^2\eta_{xx}-\frac{5}{3}  \hat{v}_x^3\eta_x-  |\hat{v}_x|^2 \hat{v} \eta_{xx} \right) \\
&\leq -c (t)\int_{-1}^1  \hat{v} (t) \eta(t)  - \frac{c^2 (t)}{2} \int_{-1}^1  \eta(t)  + \int_{-1}^t \int_{-1}^1 \frac{1}{2}(\eta_t-\eta_{xxxx})  (2 c \hat{v} +c^2) - \int_{-1}^t \int_{-1}^1 c  |\hat{v}_x|^2 \eta_{xx}  \\
& = -c (t)\int_{-1}^1  {v} (t) \eta(t) + \frac{c^2 (t)}{2} \int_{-1}^1  \eta(t) + \int_{-1}^t \int_{-1}^1 \frac{1}{2}\eta_t  (2 c {v} -c^2) - \int_{-1}^t \int_{-1}^1 (c\eta)_{xxxx}  {v} \\
&\hspace{10.5cm}- \int_{-1}^t \int_{-1}^1 c  |{v}_x|^2 \eta_{xx},
\end{aligned}
\end{equation}
where we substituted $v = \hat{v} + c(t)$ and used the fact that $ \int_{-1}^1 \eta_{xxxx} =0$ for the last line. 

In order to deal with the second and the third term of the last line of \eqref{eq:pcac1}, let us observe that for every $t$ 
\begin{equation}\label{eq:adc1}
c (t) \int_{-1}^1 \eta (t)  =  \frac{\int_{-1}^1 v(t)\varphi_{0}}{\int_{-1}^1 \varphi_{0}} \chi_0 (t) \int_{-1}^1 \varphi_{0}  = \int_{-1}^1 \eta (t) v (t),
\end{equation}
by the definition of $c(t)$ and the fact that $\eta (x,t) =  \chi_0 (t) \varphi_{0} (x)$. Therefore, first integrating by parts and then using \eqref{eq:adc1}
\[
\frac{c^2 (t)}{2} \int_{-1}^1  \eta(t) + \int_{-1}^t \int_{-1}^1 \frac{1}{2}\eta_t (2 c {v} -c^2) = 
 \int_{-1}^t \int_{-1}^1 \eta_t c {v} + \int_{-1}^t  c'\,c \int_{-1}^1 \eta =  \int_{-1}^t \int_{-1}^1 (c\eta)_t v.
\]
Using the above identity in the last line of \eqref{eq:pcac1} we obtain
\begin{equation}\label{eq:pcac12}
\begin{aligned}
&\frac{1}{2} \int_{-1}^1 \hat{v}^2(t) \eta(t) + \int_{-1}^t \int_{-1}^1 |\hat{v}_{xx}|^2\eta  \\
&\hspace{2.5cm} - \int_{-1}^t \int_{-1}^1 \left(\frac{1}{2}(\eta_t-\eta_{xxxx})  \hat{v}^2 +2 |\hat{v}_x|^2\eta_{xx}-\frac{5}{3}  \hat{v}_x^3\eta_x-  |\hat{v}_x|^2 \hat{v} \eta_{xx} \right) \\
&\le -c (t)\int_{-1}^1  {v} (t) \eta(t)  + \int_{-1}^t \int_{-1}^1 ((c\eta)_t-(c\eta)_{xxxx})  {v}  - \int_{-1}^t \int_{-1}^1 c | {v}_x|^2 \eta_{xx} =:I.
\end{aligned}
\end{equation}

We will show that $I$ vanishes. To this end let us observe that we can use $\phi(x,t) \coloneqq c(t) \eta (x,t)$ as a test function in \eqref{weak_form_cylinder_with_cutoff}. Indeed, via Remark \ref{rem:ct} $\p_t (c\eta ) \in L_{5/3,\infty } (Q)$ and $c(t)$ is continuous. This gives $I=0$, and so \eqref{eq:pcac12} reduces to 
\[\begin{split}
 \frac{1}{2} \int_{-1}^1 \hat v^2(t) \eta(t) + &\int_{-1}^t \int_{-1}^1 |\hat v_{xx}|^2\eta \\
 &\leq \int_{-1}^t \int_{-1}^1 \left(\frac{1}{2}(\eta_t-\eta_{xxxx})  \hat  v^2
+2 |\hat  v_x|^2\eta_{xx}-\frac{5}{3}  \hat  v_x^3\eta_x- | \hat  v_x|^2 \hat  v \eta_{xx} \right)
\end{split}
\]
  for every $t \in (-1,0)$. Thus

\[
 \begin{aligned}
\sup_{t \in (- 1, 0)}  \int_{-1}^{1}   \hat v(t)^2 \eta (t) + \int_{Q} | \hat v_{xx}|^2 \eta &\leq C 
\int_Q \left( \hat v^2 + |\hat v_{x}|^2 +   | \hat v_{x}|^3 +| \hat v_x|^2 | \hat v |  \right) \ \\
&\le C \int_{Q}  \left( \hat v^2 + |\hat v |^3+| \hat v_{x}|^2 +   | \hat v_{x}|^3 \right) \\
&\leq C \int_Q  \left( |\hat v_{x}|^2 +   | \hat v_{x}|^3 \right),
\end{aligned}
\]
as required, where we used the spatial Poincar\'e-Sobolev inequality (for functions with vanishing $\varphi$-related means) in the last line. 
\end{proof}
\section{Higher integrability}
Here we prove part (i) of Theorem \ref{thm_main}. First we derive a reverse H\"older inequality from the Caccioppoli inequality \eqref{eq:cacc}. Indeed, \eqref{eq:cacc} gives 
\begin{equation}\label{eq:cacchi}
\begin{aligned}
\sup_{t \in (t_0 - R^4, t_0)}  \int_{x_0-R}^{x_0+R} |\hat v \eta_{z_0,  R}|^2 \d x + \int_{Q_R (z_0)} |\p_{xx} (\hat v \eta_{z_0, R} )|^2 &\le \frac{C}{R^2}  \int_{Q_R (z_0)} |\hat v_{x}|^2 + \frac{C}{R}  \int_{Q_R (z_0)} |\hat v_{x}|^3 \\
& \hspace{-1cm} \le \frac{C}{R^\frac{1}{3}}  \left(\int_{Q_R (z_0)} |\hat v_{x}|^3\right)^\frac{2}{3} + \frac{C}{R}  \int_{Q_R (z_0)} |\hat v_{x}|^3
\end{aligned}
\end{equation}
Recall that by assumption \eqref{eq:bass}
\[
\sup_{Q_R (z_0) \subset Q} \frac{1}{R^2}  \int_{Q_R (z_0)} |v_x|^3 \le M.
\]
Thus \eqref{eq:cacchi} yields
\begin{equation}\label{eq:cacchi2}
\sup_{t \in (t_0 - R^4, t_0)}  \int_{x_0-R}^{x_0+R} |\hat v \eta_{z_0,  R}|^2 \d x+ \int_{Q_R (z_0)} |\p_{xx} (\hat v \eta_{z_0,  R} )|^2 \le  C \frac{1+ M^\frac{1}{3}}{R^\frac{1}{3}}  \left(\int_{Q_R (z_0)} |\hat v_{x}|^3\right)^\frac{2}{3} 
\end{equation}
for every $Q_R (z_0)\subset Q$. Let us use multiplicative inequality \eqref{eq:mhi} for $\hat v \,\eta_{z_0,R}$ to compute 
\[
\begin{split}
\left( \int_{Q_{R/2}(z_0)} | \hat v_x |^{10/3} \right)^{3/5}& \leq C \| \hat v \eta_{z_0,R} \|_{L_{2,\infty } (Q_{R}(z_0))}^{4/5}\| \p_{xx} (\hat v \eta_{z_0,R} ) \|_{L_{2,2 } (Q_{R}(z_0))}^{6/5} \\
&\leq  \sup_{t \in (t_0 - R^4, t_0)}  \int_{x_0-R}^{x_0+R} |\hat v \eta_{z_0,R} |^2\d x + \int_{Q_{R} (z_0)} |\p_{xx} ( \hat v \eta_{z_0,R} )  |^2 \\
&\leq (1+M^{1/3} ) \frac{1}{R^\frac{1}{3}}  \left(\int_{Q_R (z_0)} |\hat v_{x}|^3\right)^\frac{2}{3} ,
\end{split}
\]
where we used \eqref{eq:cacchi2} in the last line. In other words we obtain the reverse H\"older inequality
\[
 \left(  \fint_{Q_\frac{R}{2} (z_0)} |\hat v_{x}|^\frac{10}{3}\right)^\frac{3}{10} \le C_M  \left(  \fint_{Q_R (z_0)} |\hat v_{x}|^3\right)^\frac{1}{3}
\]
for every $Q_R(z_0)\subset Q$. Applying the Gehring Lemma (see Proposition 1.3 in \cite{GiaquintaStruwe82}, for example) gives part (i) of Theorem \ref{thm_main}.

\section{$\varepsilon$-regularity}
In this section we prove the part (ii) of Theorem \ref{thm_main}. It relies on quantifying decay of 
\[
Y(r,v)\coloneqq \left( \fint_{Q_r} \left| v_x - (v_x)_{r} \right|^3 \right)^{1/3}
\]
as $r\to 0^+$, where $(f)_{r} \coloneqq \fint_{Q_r} f$. The needed decay lemma is stated and proved in subsection \ref{ssec:dec}. Then we conclude our proof of the part (ii) of Theorem \ref{thm_main} in subsection \ref{ssec:prc}, using the Campanato characterisation of H\"older continuity.

\subsection{Decay estimate}\label{ssec:dec}

\begin{lemma}[Decay Lemma]\label{lem_iteration} For any $M>1$ and $\theta \in (0,1/4)$ there exists $\varepsilon_0 (\theta, M) >0$ such that for any suitable weak solution $v$ to the SGM \eqref{SGMe} on $Q_1$, if 
\begin{equation}\label{ass:dl}
\left| (v_x)_1 \right| \leq M \quad \text{ and }\quad  Y(1,v)\leq \varepsilon_0,
\end{equation}
then
\[
Y(\theta ,v) \leq 6cM \theta \,Y(1,v),
\]
where $c$ is the universal constant from Lemma \ref{lem_campanato_linear}.
\end{lemma}

In the following Sections \ref{sec_step0}-\ref{sec_step3} we prove Lemma \ref{lem_iteration}: we first compare nonlinear problem \eqref{SGMe} to a linear one \eqref{eq:linb} using the blow-up technique, and then we use quantitative decay for a linear system, by means of Lemma \ref{lem_campanato_linear}. 

\subsubsection{Proof of Lemma \ref{lem_iteration}. The Setup}\label{sec_step0}

Suppose that Lemma \ref{lem_iteration} is false. Then there exist numbers $\theta \in (0,1/4)$, $M>1$, a sequence $\varepsilon_k \to 0$ and a sequence of suitable weak solutions $v^{(k)}$ such that 
 \eqnb\label{ass:dla}
\left| ( v^{(k)}_x )_1 \right| \leq M,\quad  Y(1,v^{(k)})=\varepsilon_k, \quad \text{ and } \quad Y(\theta , v^{(k)}) > 6cM \theta \varepsilon_k 
 \eqne
 We let
\[
 u^{(k)} \coloneqq \varepsilon_k^{-1} \left( v^{(k)} - [v^{(k)}]_{1} - (\p_x v^{(k)} )_1\,x \right),
 \]
where we set  $[f]_{1} \coloneqq \fint_{-1}^{0} \langle f(t)\rangle_{0,1}  \d t$, with $\langle f(t)\rangle_{0,1} $ denoting the $\varphi$-related mean (recall \eqref{eq:mmv}).
Since $\varphi$ is even we have $[u^{(k)}]_{1}=( u^{(k)}_x )_1=0$; $u^{(k)}$ also normalises \eqref{ass:dla}, i.e.\
 \eqnb\label{prop_of_uk}
 Y(1,u^{(k)} ) = 1 \quad \text{ and }\quad Y(\theta , u^{(k)} )> 6cM \theta.
 \eqne 
 
 \subsubsection{Proof of Lemma \ref{lem_iteration}. Uniform estimate on the rescalings $u^{(k)}$}\label{sec_step1}
 
Here we show that
\eqnb\label{EI_u_k}
\|u^{(k)} \eta_{0, 1} \|_{L_{2, \infty} (Q_{1})} + \| \p_{xx}(u^{(k)}\eta_{0,  1} ) \|_{L_{2,2}(Q_{1})} \leq C_M.
\eqne
uniformly in $k$. (Recall \eqref{eq:def_eta} for the definition of $\eta_{0,  1}$.)

Letting $\beta_k\coloneqq (v^{(k)}_x)_1$, we see that $u^{(k)}$ satisfies the following {\em{perturbed SGM}}
 \eqnb\label{eq_for_uk}
 \p_t u^{(k)} +  u^{(k)}_{xxxx} + \varepsilon_k \p_{xx} (u^{(k)}_x)^2 + 2 \beta_k u^{(k)}_{xxx} =0
 \eqne
in $Q_1$ in the sense of distributions. Moreover, the local energy inequality \eqref{LEI_alt_form} for $v^{(k)}$ gives that 
\eqnb\label{eq_for_uk_lei}\begin{split}
&\frac{1}{2} \int_{-1}^1 (u^{(k)})^2(t) \phi(t) + \int_{-1}^t \int_{-1}^1 | u^{(k)}_{xx}|^2\phi \\
& \leq\varepsilon_k \int_{-1}^t \int_{-1}^1 \left(\frac{1}{2}(\phi_t-\phi_{xxxx})(u^{(k)})^2 +2( u^{(k)}_x)^2\phi_{xx}-\frac{5}{3}(u_x^{(k)})^3\phi_x-(u_x^{(k)})^2u^{(k)}\phi_{xx} \right) \\
&- \beta_k \int_{-1}^t \int_{-1}^1 3 (u_x^{(k)})^2\phi_{x}  - (u^{(k)})^2\phi_{xxx} 
\end{split}
\eqne
 for any nonnegative $ \phi \in C_0^\infty (Q_1)$ and every $t \in (-1, 0)$. Here, as before, we skipped the variable under the integrals below, and every integral ``$\int_{-1}^1$'' is taken with respect to $x$ and we wrote argument ``$(t)$'' to point out that the integral is taken at a given time $t$. Unless specified otherwise, we will apply the same convention in what follows. Letting 
 \[\hat u^{(k)} \coloneqq u^{(k)} - \langle u^{(k)}(t)\rangle_{0, 1} 
 \]
 we can use the above inequality  to obtain the following Caccioppoli inequality for $\hat u^{(k)}$, analogously to how \eqref{eq:cacc} was proven
 \begin{equation}\label{eq:caccp}
\begin{aligned}
\sup_{t \in (- 1, 0)}  \int_{-1}^{1} \hat u^{(k)}(t) ^2 \eta_{0,  1} (t) +& \int_{Q_1} |  \hat u^{(k)}_{xx}|^2 \eta_{0, 1} \\ \le& {\varepsilon_k C}  \int_{Q_1} |  \hat u_{x}^{(k)}|^2 + {\varepsilon_k C}  \int_{Q_1} | \hat u_{x}^{(k)}|^3
+{C |\beta_k| }  \int_{Q_1 } | \hat u_{x}^{(k)}|^2
\end{aligned}
\end{equation}
for each $k$. Recalling that $\varepsilon_k \leq 1$, $| \beta_k | = \left| ( v^{(k)}_x )_1 \right| \leq M$ (see \eqref{ass:dla}) and observing that \eqref{prop_of_uk} gives
\[ \int_{Q_1} |  \hat u_x^{(k)} |^3  =  \int_{Q_1} | u_x^{(k)} |^3  = 2Y(1,u^{(k)} )^3 =2
\]
we obtain
\[
\|\hat u^{(k)} \eta_{0,  1} \|_{L_{2, \infty} (Q_{1})} + \| \p_{xx}( \hat u^{(k)} \eta_{0,  1}) \|_{L_{2,2}(Q_{1})} \leq C_M.
\]
Comparing this with \eqref{EI_u_k}, we see that it suffices to show that $\sup_{t\in (-1,0)} \langle u^{(k)} (t) \rangle_{0,1}  \leq C_M$. We will write 
\[
c(t) \coloneqq \langle u^{(k)}(t) \rangle_{0,1} ,
\]
for brevity. Similarly as in \eqref{eq:mmv1}  we have  for every $\psi \in C_0^\infty (-1,0)$
\[\begin{split}
\left| \int_{-1}^0 c (t) \psi' (t) \, \d t \right|  &\leq  C  \|  u^{(k)}_x \|_{L_{3,3} (Q_1)}^2 \| \psi \|_{L_{3} (-1,0)} + C \|  u_x^{(k)} \|_{L_{3,3} (Q_1)} \| \psi \|_{L_{3/2} (-1,0)} \\
&\leq C \| \psi \|_{L_3 (-1,0)},
\end{split}
\]
 where we used \eqref{prop_of_uk}. Thus $c(t)$ is continuous with $c'(t) \in L_{3/2} (-1,0)$. Since $\int_{-1}^0 c(t)\, \d t =[u^{(k)}]_{1}=0$ there exists $t_0\in (-1,0)$ such that $c(t_0)=0$, and hence
 \[
 |c(t)| = |c(t) - c(t_0)| = \left| \int_{t_0}^t c'(s) \, \d s \right| \leq  \| c' \|_{L_{3/2}(-1,0)} \leq C
 \] 
 for every $t\in (-1,0)$, as required. We thus obtained \eqref{EI_u_k}. 
 
 \subsubsection{Proof of Lemma \ref{lem_iteration}. Blowup limit equation}\label{sec_blowup_limit_eq}
 Here we extract a sequence of $u^{(k)}$ converging to a limit $u$ that satisfies a linear equation and that $Y(\theta , u ) < 6cM \theta$.
 
 Indeed, from \eqref{prop_of_uk}, the interpolation inequality $\| f \|_{L_{10,10} (U )} \leq C \| f \|^{\frac45 }_{L_{2,\infty } (U )} \| f \|^{\frac15 }_{W_{2,2 }^{2,0} (U )} $  applied to $u^{(k)}\eta_{0,1}$ and \eqref{EI_u_k} we obtain 
 \eqnb\label{L3_bounds_uk}
 \int_{Q_{1}} | u_x^{(k)}|^3 = 2 ,\qquad \int_{Q_{1/2}} (u^{(k)})^{10} \leq C_M.
 \eqne
This and the fact that $| \beta_k | \leq M$ allow to extract a subsequence (which we relabel) such that
\[
 u^{(k)} \rightharpoonup u\quad  \text{ in } W^{1,0}_{3,3} (Q_{1/2}) \qquad \text{ and }\qquad  \beta_k  \to \beta
\]
for some $u\in W^{1,0}_{3,3} (Q_{1/2})$ and $\beta \in [-M,M]$. Since also $\varepsilon_k \to 0$, we can take $k\to \infty $ in the distributional formulation of {perturbed SGM} \eqref{eq_for_uk} to obtain that $u$ is a distributional solution to the linear equation
 \[
 u_t + u_{xxxx} + 2\beta u_{xxx} =0\quad \text{ in }Q_{1/2}.
 \]
 Applying Lemma \ref{lem_campanato_linear} and using the fact that $\int_{Q_{\theta }} |u_x|^3 \leq 2 $ for every $\theta <1/2$ (a consequence of the weak convergence and \eqref{L3_bounds_uk}) we obtain 
 \begin{equation}\label{eq:lwsc}
 Y(\theta , u) = \left( \fint_{Q_\theta} \left| u_x - (u_x)_{\theta} \right|^3 \right)^{1/3} \leq c(1+ 2|\beta |) \theta \| u_x \|_{L_3(Q_{1})}  < 6cM  \theta,
 \end{equation}
 as required.

 \subsubsection{Proof of Lemma \ref{lem_iteration}. Compactness and contradiciton}\label{sec_step3}
Here we will show that (on a subsequence)
\begin{equation}\label{eq:strongc}
 u_x^{(k)} \to  u_x \qquad \text{ in } \quad L_3 (Q_{1/2}) .
 \end{equation}
 This will conclude the proof of Lemma \ref{lem_iteration} since the strong limit \eqref{eq:strongc} implies $Y(\theta , u^{(k)} ) \to Y(\theta , u)$ and so the inequality in \eqref{prop_of_uk} yields
\begin{equation}\label{eq:contra}
 6cM  \theta \le  Y(\theta , u) \stackrel{\eqref{eq:lwsc}}{<} 6cM  \theta,
\end{equation}
 a contradiction.
 
In order to justify \eqref{eq:strongc} we will use an Aubin-Lions argument. Applying \eqref{EI_u_k} and the interpolation inequality $\| f \|_{W^{7/6,0}_{3 , 3} }\leq C \| f \|^{\frac13 }_{L_{2,\infty } } \| f \|^{\frac23 }_{W_{2,2 }^{2,0} }$ to the function $u^{(k)} \eta_{0,1}$ we obtain that
\eqnb\label{u_k_for_AL}
\|  u^{(k)} \|_{W^{7/6,0}_{3 , 3} (Q_{1/2})}  \leq C_M.
\eqne
On the other hand for every $\phi \in C_0^\infty (Q_{1/2})$ 
\[\begin{split}
\left| \int_{Q_{1/2}}  u^{(k)}\, \p_t \phi \right| &= \left| -\int_{Q_{1/2}} ( u_{xx}^{(k)}) \left( \phi_{xx} - 2\beta_k \phi_x \right) - \varepsilon_k \int_{Q_{1/2}} \left( u_x^{(k)} \right)^2 \phi_{xx}\right|\\
&\leq C_M \| \phi \|_{W^{2, 0}_{3,3} (Q_{1/2})} \left( \| u_{xx}^{(k)} \|_{L_{3/2,3/2} (Q_{1/2})} +\|  u_x^{(k)} \|^2_{L_{3,3} (Q_{1/2})}  \right)\\
&\leq  C_M \| \phi \|_{W^{2, 0}_{3,3} (Q_{1/2})},
\end{split}
\]
where we used \eqref{EI_u_k} and \eqref{L3_bounds_uk} in the last line.
Hence
\eqnb\label{tu_k_for_AL}
\| u^{(k)}_t \|_{L_\frac{3}{2} ( -1/16,0; (W^2_3(B_{1/2}))^*)} \leq C_M.
\eqne 
This, \eqref{L3_bounds_uk} and \eqref{u_k_for_AL} allow us to use the Aubin--Lions lemma (see Section 3.2.2 in \cite{temam}, for example) to extract a subsequence of $u^{(k)}$ that converges strongly in ${W^{\frac{7}{6}-\delta, 0}_{3,3} (Q_{1/2})}$ for any fixed $\delta \in (0,7/6)$. (Recall that ${W}{^{\frac{7}{6}-\delta,0}_{3,3}}(Q_r) = L_3(t-r^4,t;{W}^{\frac{7}{6}-\delta}_3  (x-r,x+r) )$.) Taking $\delta \coloneqq 1/6$ gives \eqref{eq:strongc}, as required. 

\subsection{Concluding $\varepsilon$-regularity proof}\label{ssec:prc}
Here we finish the proof of part (ii) of Theorem \ref{thm_main} by iterating Lemma \ref{lem_iteration}. Indeed we have the following.
\begin{corollary}\label{cor_iterate_in_theta} Given $\gamma \in (0,1)$ and $M\geq 1$, there exist $\varepsilon_0, \theta \in (0,1/2)$ with the following property. If $v$ is a suitable weak solution to the SGM \eqref{SGMe} on $Q_1$ such that
\[
\left| (v_x )_1 \right| \leq M,\qquad Y(1,v)\leq \varepsilon_0
\]
then
\eqnb\label{claim_of_cor}
\theta^{k-1} \left| (v_x)_{\theta^{k-1}}\right|\leq M\quad \text{ and }\quad Y(\theta^k , v)\leq \theta^\gamma Y(\theta^{k-1} , v)
\eqne
for every $k\geq 1$.
\end{corollary}
\begin{proof}
We fix $\theta \in (0,1/2)$ so small that
\begin{equation}\label{iter_c1}
6cM \theta^{1-\gamma} <1.
\end{equation}
(Recall that $c$ is the universal constant from Lemma \ref{lem_campanato_linear}.) Let $\varepsilon_0$ be sufficiently small so that Lemma \ref{lem_iteration} holds and
\begin{equation}\label{iter_c2}
\varepsilon_0 < \theta^5 M/2.
\end{equation}
The case $k=1$ follows from our assumptions and Lemma \ref{lem_iteration}. For $k>1$ we proceed by induction. Suppose that \eqref{claim_of_cor} holds for $k'\leq k$. Then
\eqnb\label{means_bound_k_step}
\begin{split}
\theta^k \left| (v_x)_{\theta^{k}}\right| &\leq \theta^k \left| (v_x)_{\theta^{k}} -(v_x)_{\theta^{k-1}} \right| + \theta^k  \left| (v_x)_{\theta^{k-1}}\right| = \theta^k \left| \fint_{Q_{\theta^k}} \left( v_x -(v_x)_{\theta^{k-1}} \right) \right| + \theta^k  \left| (v_x)_{\theta^{k-1}}\right|\\
\leq &\,\theta^{k-5}  \fint_{Q_{\theta^{k-1}}} \left| v_x -(v_x)_{\theta^{k-1}}  \right| + \theta \, M \leq \theta^{k-5} \left( \fint_{Q_{\theta^{k-1}}} \left| v_x -(v_x)_{\theta^{k-1}} \right|^3 \right)^{1/3} +  M/2\\
= &\, \theta^{k-5} Y(\theta^{k-1} , v) + M/2 \leq \theta^{-5} Y(1,v) + M/2 \leq  \theta^{-5} \varepsilon_0 + M/2 \leq  M,
\end{split}
\eqne
where we used Jensen's inequality, the fact that $\theta<\frac12$ (so that in particular $\theta^\gamma <1$ and $\theta^k<1$), the inductive assumption (for $k'=1,\ldots, k-1$), and the choice \eqref{iter_c2}.

It remains to show that $Y(\theta^{k+1}, v) \leq \theta^\gamma Y(\theta^{k},v)$. To this end let us rescale
\[ v^k (x,t) \coloneqq v(\theta^k x,\theta^{4k} t).\]
In particular, $v^k$ is a suitable weak solution of SGM \eqref{SGMe} on $Q_1$, and
 \[
 \left| ( v_x^k )_1 \right| = \theta^k   \left| (v_x)_{\theta^{k}}\right| \leq M \quad\text{ and }\quad
 Y(1,v^k) = \theta^k Y(\theta^k , v) \leq Y(1,v) \leq \varepsilon_0,
 \]
where we used \eqref{means_bound_k_step}, the assumption \eqref{claim_of_cor} (for $k'=1,\ldots , k$) and the fact that $\theta <1$. Thus Lemma \ref{lem_iteration} gives $Y(\theta , v^k ) \leq 6cM  \theta \,\,Y(1,v^k)$, from which we conclude
\[
Y(\theta^{k+1} , v) = \theta^{-k } Y(\theta , v^k ) \leq 6cM  \theta^{1-k} \,Y(1 ,v^k ) = 6cM  \theta\, Y(\theta^k , v) < \theta^\gamma \,Y(\theta^k , v),
\]
via the choice \eqref{iter_c1}.
\end{proof}
We can now conclude the proof of part (ii) of Theorem \ref{thm_main}. Without loss of generality we assume that $Q=Q_1(0,0)$. Recall that we need to show that for any $\gamma \in (0,1)$ there exists $\varepsilon >0$ such that $I(Q)\equiv \int_Q | v_x |^3 \leq \varepsilon$ implies that
\eqnb\label{to_show_conclusion}
\begin{aligned}
&\left| v_x (x,t) - v_x (y,s) \right| \leq C_\gamma \varepsilon^{1/3}  \left( |x-y| + |t-s |^{1/4} \right)^\gamma \quad \text{ for } (x,t),(y,s)\in Q_{\frac{1}{2}}.
\end{aligned}
\eqne

We first deduce from Corollary \ref{cor_iterate_in_theta} that
\eqnb\label{to_show_cor_holder_cnty}
\left( \fint_{Q_r(y,s)} \left| v_x - (v_x)_{Q_r(y,s)} \right|^3 \right)^{1/3}\leq C_\gamma \varepsilon^{1/3} r^\gamma
\eqne
for every $(y,s)\in Q_{1/2}$, $r\in (0,1/2)$, where $(v_x)_{Q_r(y,s)}\coloneqq \fint_{Q_r (y,s)} v_x$.

Indeed, let $\varepsilon_0, \theta \in (0,1/4)$ be given by Corollary \ref{cor_iterate_in_theta} applied with $M\coloneqq 1$, and let $\varepsilon \coloneqq  \varepsilon_0^{3}/16 $ and  
\[
u(x,t) \coloneqq v(y+x/2,s+t/16).
\]
Given $r\in (0,1/2)$ let $K \in \NN $ be such that
\[
\theta^{K+1} \leq r < \theta^{K}.
\]
By assumption \eqref{eq:sass} and Jensen's inequality
\[\begin{split}
\left| (u_x )_1 \right| &\leq \left( \fint_{Q_1} |u_x |^3 \right)^{1/3} = 2^{1/3}  \left( \int_{Q_{1/2}(y,s)} |v_x |^3 \right)^{1/3} \leq  2^{1/3} I(Q)^{1/3} \leq 2^{1/3} \varepsilon^{1/3} \leq 1 =M ,\\
 Y(1,u)&= \left( \fint_{Q_1} |u_x - (u_x)_1 |^3 \right)^{1/3} \leq 2 \left( \fint_{Q_1} |u_x |^3 \right)^{1/3} \leq 2^{4/3} \varepsilon^{1/3} =\varepsilon_0,
 \end{split}
\]
and so Corollary \ref{cor_iterate_in_theta} gives
\[
Y(\theta^K , u ) \leq \theta^{\gamma K } Y(1,u) \leq  \theta^{\gamma K } \varepsilon_0.
\]
Noting that $r/\theta^{K} \in (\theta ,1 )$ and that $| (u_x)_r - (u_x)_{\theta^K} | \leq  \theta^{-5} Y(\theta^K,u) $ (as in \eqref{means_bound_k_step} above) we obtain 
\[\begin{split}
\left( \fint_{Q_r} \left| u_x - (u_x)_r \right|^3 \right)^{1/3}  &\leq \left( \fint_{Q_r} \left| u_x - (u_x)_{\theta^K} \right|^3 \right)^{1/3} + \theta^{-5} Y(\theta^K,u) \\
&\leq (\theta^{-5/3} + \theta^{-5}  ) Y(\theta^K , u) \leq (\theta^{-5/3} + \theta^{-5}  ) \theta^{\gamma K} \varepsilon_0 \\
&\leq (\theta^{-5/3} + \theta^{-5}  ) \theta^{-\gamma } r^\gamma \varepsilon_0 =C \varepsilon^{1/3} r^\gamma
\end{split}\]
for every $r\in (0,1)$, where $C=C (\gamma ) $. (Recall that $\theta$ is fixed depending on $\gamma $, see \eqref{iter_c1}.) The claim \eqref{to_show_cor_holder_cnty}, follows by writing the above inequality in terms of $v$.\\

Using \eqref{to_show_cor_holder_cnty} we obtain \eqref{to_show_conclusion} by applying the following Campanato Lemma with $p=3$, $R=1$, cf.\ the original \cite{campanato} or Lemma A.2 in \cite{SGM}.
\begin{lemma}[Campanato]\label{Campanato}
  Let $R\in(0,1]$, $f\in L_{1,1}(Q_R(0))$ and suppose that there exist positive constants $\gamma \in(0,1]$, $N>0$, such that
  \[
  \left(\fint_{Q_r(z)}|f(y)-(f)_{Q_r(z)}|^p\,\d y\right)^{1/p}\le Nr^{\gamma}
  \]
  for any $z\in Q_{R/2}(0)$ and any $r\in(0,R/2)$, where $(f)_{Q_r(z)}\coloneqq \fint_{Q_r (z)} f$. Then $f$ is H\"older continuous in $Q_{R/2}(0)$ with 
  \[
  |f(x,t)-f(y,s)|\le cN(|x-y|+|t-s|^{1/\alpha})^\gamma
  \]
for all $(x,t),(y,s)\in Q_{R/2}(0)$.
\end{lemma}

\section{Appendices}

\subsection{Appendix on weak solutions}\label{sec_app_on_weak_sol}

\begin{proof}[Proof of Remark \ref{rem:str}]
The distributional formulation \eqref{weak_form_R} yields
\[
\left|\int_0^{\infty} \int_\RR v\, \phi_t \right|  \leq \int_0^{\infty} \left( \| v_{xx} (t)\|_{2} \| \phi_{xx} (t) \|_{2} + \| v_x (t) \|^2_{10/3} \| \phi_{xx} (t) \|_{5/2} \right) \d t.
\]
In particular, for $\phi (x,t) = \varphi (x) \chi (t)$, where $\chi \in C_0^\infty ((0,\infty ))$ and $\phi \in C_0^\infty (\RR)$, using \eqref{eq:mhi}, and the energy inequality \eqref{EI_on_R} 
\[
\left|\int_0^{\infty } \int_\RR v\, \varphi \chi_t \right|  \leq C_{u_0} \|  \varphi \|_{W^{5/2}_2}  \|  \chi \|_{L_\frac52},
\]
i.e.\ $v_t \in L_\frac53 ((0,\infty ); (W^{5/2}_2 (\RR))^*)$. Hence redefining $v$ on a set of measure zero we have $v \in C((0,T); (W^{5/2}_2 (\RR))^*)$. Since also $v \in L_{2,\infty}$, we have weak $L_2$ continuity of $v$, via e.g.\ Lemma 2.2.5.\ of \cite{Pokorny_notes}, which shows (iii).

We now verify that the energy inequality \eqref{EI_on_R} holds for every $t>0$. Choose any $t\geq 0$ and a sequence $t_n \to t^+$ such that the energy inequality holds at each $t_n$. Since $v(t_n) \rightharpoonup v(t)$ (by (iii)) we can take $\liminf_{t_n\to t}$ of the energy inequalities and use lower weak semicontinuity of the norm to write $\| v(t) \|^2 \leq \liminf_{t_n \to t^+} \| v(t_n ) \|^2 \leq \| v_0 \|^2 - 2 \int_0^t \int_{\RR } | v_{xx} |^2 $.

As for (iv), given $\phi\in C_0^\infty ((-1,\infty ) \times \RR )$ we can multiply $\phi$ by a cutoff in time (as in Lemma 3.14 in \cite{NSE_book}, for example) to obtain \eqref{weak_form_R_with_cutoff} for almost every $t>0$. Weak continuity (iii) guarantees that \eqref{weak_form_R_with_cutoff} holds for every $t>0$.

As for (v), note that (iv) implies weak convergence $v(t) \rightharpoonup v_0$ as $t\to 0^+$. Moreover $\| v(t) \| \to \| v_0 \|$ since $\liminf_{t\to 0^+} \| v(t) \|  \geq \| v_0 \| $ (from the weak convergence) and $\limsup_{t\to 0^+} \| v(t) \|  \leq \| v_0 \| $ (from the energy inequality \eqref{EI_on_R}), and so strong convergence follows, cf.\ p.106 in \cite{Seregin_notes}.
\end{proof}

\begin{proof}[Proof of Theorem \ref{thm:Cauchy}]
Given $l>0$ let $v_{0,l} \in C_0^\infty (-l,l)$ be such that $\| v_{0,l} \| \leq \| v_0 \|$ and $v_{0,l}\to v_0$ strongly in $L^2$ as $l\to \infty $, and let $T_l >0$ be such that $T_l \to \infty$ as $l\to \infty $. Let us denote by $\mathring W^2_2(-l,l)$ the completion of $C_0^\infty(-l,l)$ in the $W^2_2$ norm.
By a straightforward modification of the arguments in Theorem 4.3 in \cite{SteWin:05} and Theorem 2.4 in \cite{SGM}, for each $l$ there exists a suitable weak solution of the initial-boundary value problem
\[
\begin{cases}
&\p_t v_l + \p_{xxxx} v_l + \p_{xx} (v_l)^2 =0 \quad \text{ on } (-l,l)\times (0,T)\\
&v_l (0) = v_{0,l}\\
&v_l= \p_x v_l =0 \quad \text{ on } \{ -l,l \} \times (0,T),
\end{cases}
\]
namely there exists $v_l\in L_\infty ((0,T_l );L_2 (-l,l) ) \cap L_2 ((0,T_l); \mathring  W^2_2 (-l, l) )$ such that, after null-extending $v_l$ from $(-l,l)$ to $\RR$, we have 
\eqnb\label{weak_form_interval}
\int_0^\infty \int \left( v_l \phi_t - (\p_{xx} v_l) \phi_{xx} - (\p_x v_l)^2 \phi_{xx} \right) = \int v_{0,l}\, \phi (0)
\eqne
for every $\phi \in C_0^\infty ((-1,T_l)\times (-l,l) )$, and
\eqnb\label{LEI_interval}
 \int_0^\infty  \int (\p_{xx} v_l)^2 \phi \leq \int_0^\infty  \int \left( \frac{1}{2} ( \phi_t - \phi_{xxxx} ) v_l^2 +2 (\p_x v_l)^2 \phi_{xx} -\frac{5}{3} (\p_x v_l)^3 \phi_x - (\p_x v_l)^2 v_l \phi_{xx} \right) 
\eqne
for every $\phi \in C_0^\infty ((-1,T_l) \times (-l,l) )$ with $\phi \geq 0$, and
\eqnb\label{EI_for_uL}
\| v_l(t) \|^2 + 2\int_0^t \| \p_{xx} v_l (\tau ) \|^2 \d \tau \leq \| v_{0,l} \|^2 \leq \| v_{0} \|^2
\eqne
for almost every $t>0$. Here we wrote $\| \cdot \| \equiv \| \cdot \|_{L_2 (\RR )}$ for brevity. Consequently, there exists $v$ such that 
\eqnb\label{weak_conv_temp}
\begin{split}
v_l \stackrel{\ast }{\rightharpoonup }& v \qquad\hspace{0.3cm} \text{ in } L_\infty ((0,\infty ); L_2 )  \\
\p_{xx}v_l {\rightharpoonup }& v_{xx} \qquad \text{ in } L_2 ((0,\infty ); L_2 ) 
\end{split}
\eqne
as $l\to \infty$ and, via lower weak semi-continuity
\eqnb\label{EI_limit}
\| v(t) \|^2 + 2\int_0^t \| v_{xx} (\tau ) \|^2\d \tau  \leq \| v_{0} \|^2
\eqne
for almost every $t>0$, in particular $v \in L_\infty(0, \infty; {L}_2 (\RR)) \cap L_2 (0, \infty;   W^2_2(\RR) )$. The multiplicative inequality \eqref{eq:mhi} controls via \eqref{EI_limit} the nonlinear term, and thus for any $\phi \in C_0^\infty ((-1,\infty)\times \RR )$ we can pass to the limit in \eqref{weak_form_interval}. We obtained $v$ satisfying the definition of energy weak solution to the Cauchy problem \eqref{SGMC}.

Let us show that $v$ is a suitable weak solution. Fix any bounded cylinder $Q_z(r) \subset (-L,L) \times (0,T)$. Identity \eqref{weak_form_interval}, the energy bound \eqref{EI_for_uL}, and the reasoning as in proof of Remark  \ref{rem:str}  imply
\eqnb\label{ut_bound}
 \| v_t \|_{L_\frac53( 0,T; (W^{2}_{5/2} (-L,L))^*)} \d  \tau  \leq C_{T,L, \| v_0 \|}  ,
\eqne
with the only difference now being $C(L)$, due to the H\"older inequality $\| v_{xx} (t)\|_{L_\frac53 (-L,L)} \le (2L)^\frac15 \| v_{xx} (t)\|_{L_2 (-L,L)}$. Applying the Aubin-Lions Lemma (see for example Theorem 2.1 on p. 184 in \cite{temam}) with $X_0\coloneqq H^2 (-L,L)$, $X \coloneqq W^{1}_\infty (-L,L)$, $X_1 \coloneqq  (W^{2}_{5/2}(-L,L))^*$, $\alpha_0\coloneqq  2$, $\alpha_1 \coloneqq 5/3$, we see that (along a subsequence)
\[
v_l \to v \quad \text{ in }L_2 ((0,T);W^{1}_{\infty } (-L,L) ). \]
Thus we can take $\lim_{l\to \infty }$ on the r.h.s.\ of \eqref{LEI_interval} and $\liminf_{l\to \infty } $ on its l.h.s.\ to obtain via l.w.s.c.\
\eqnb\label{LEI_line}
 \int_0^L \int_{-L}^L  v_{xx}^2 \phi \leq  \int_0^L \int_{-L}^L \left( \frac{1}{2} ( \phi_t - \phi_{xxxx} ) v^2 +2 v_x^2 \phi_{xx} -\frac{5}{3} v_x^3 \phi_x - v_x^2 v_l \phi_{xx} \right)
\eqne
for any $\phi \in C_0^\infty (Q)$.
Applying a cut-off procedure in time, i.e.\ rewriting \eqref{LEI_line} for $\phi (x,t) \chi (t)$, with $\chi$ being a bump function around $t$, we obtain that 
\eqnb\label{LEI_line_with_t}
 \frac{1}{2} \int v(t)^2 \phi(t) + \int_0^t  \int  v_{xx}^2 \phi \leq \int_0^t  \int \left( \frac{1}{2} ( \phi_t - \phi_{xxxx} ) v^2 +2 v_x^2 \phi_{xx} -\frac{5}{3} v_x^3 \phi_x - v_x^2 v_l \phi_{xx} \right)
\eqne
for almost every $t>0$. More precisely, for every Lebesgue point of $f(t)\coloneqq \int v(t)^2 \phi(t) $, so the times $t>0$ where \eqref{LEI_line_with_t} holds  depends on the choice of $\phi$. However, since $v(t)$ is weakly continuous in $L_2$ (by (iii), cf.\ Remark \ref{rem:str}), the same is true of $v(t)\sqrt{\phi(t)}$ and so we can extend \eqref{LEI_line_with_t} for all $t>0$. 
\end{proof}

\subsection{Proof of Lemma \ref{lem_campanato_linear}}\label{sec_proof_linear_camp}
We show now a Campanato-type estimate for the linear equation  \eqref{eq:linb}.
We may assume that $u\in C^\infty (Q_{3/4})$. Otherwise first we use the standard mollification $u^{(\varepsilon )}$ and then our claim follows from taking the limit $\varepsilon\to 0$ in \eqref{campanato_linear}.

Letting $\phi \in C_0^\infty (Q_{3/4})$ be such that $\phi=1$ on $Q_{5/8}$, multiplying \eqref{eq:linb} by $u\phi$ and integrating by parts we obtain
\eqnb\label{LEI_temp}
\int_{Q_{3/4}}  u_{xx}^2\phi=\int_{Q_{3/4}}\left( \frac{1}{2}u^2(\phi_t-\phi_{xxxx} ) + u_x^2 \left( 2 \phi_{xx}+ \frac{3\beta }{2}\phi_x \right) - \beta u_x u \phi_{xx} \right)
\eqne
Thus
\eqnb\label{u_xx_is_in_L^2}
\| u_{xx}\|_{L_{2,2} (Q_{5/8} )} \leq C_{ \beta } \left(\|u\|_{L_{2,2}(Q_{3/4})}+\|u_x\|_{L_{2,2}(Q_{3/4})}\right).
\eqne
Since any space derivative $\p_x^m u$ ($m\geq 0$) satisfies \eqref{eq:linb} on $Q_{3/4}$ and any mixed derivatives $\p_t^k \p_x^m u$ ($k,m\geq 0$) can be expressed in terms of pure space derivatives via \eqref{eq:linb} itself, we obtain 
\eqnb\label{higher_derivs_in_L^2}
\| \p_t^k \p_x^m u\|_{L_{2,2} (Q_{1/2} )} \leq C_{k,m, \beta } \left(\|u\|_{L_{2,2}(Q_{3/4})}+\|u_{x}\|_{L_{2,2}(Q_{3/4})}\right)
\eqne
for any $k\geq 0$, $m\geq 2$, by bootstrapping the inequalities of type \eqref{u_xx_is_in_L^2} on a sequence of decreasing cylinders. An embedding and \eqref{higher_derivs_in_L^2} imply
\[
\| u_x \|_{W^{1,0}_{\infty, \infty} (Q_{1/2})} \leq C_\beta  \left(\|u\|_{L_{2,2}(Q_{3/4})}+\|u_{x}\|_{L_{2,2}(Q_{3/4})}\right).
\]
Hence
\[
\begin{split}
\left( \fint_{Q_\theta} \left| u_x - (u_x)_\theta \right|^p \right)^{1/p}&\leq \max_{Q_\theta }  \left| u_x - (u_x)_\theta \right|=  \max_{Q_\theta }  \left| u_x - u_x (z_0) \right| \leq C \theta  \| u_x \|_{W^{1,0}_{\infty,\infty} (Q_{1/2})}\\
&\leq C_\beta \theta  \left(\|u\|_{L_{2,2}(Q_{3/4})}+\|u_{x}\|_{L_{2,2}(Q_{3/4})}\right)
\end{split}
\]
since by Darboux property there exists $z_0\in Q_{\theta }$ such that $u_x (z_0)=(u_x)_{\theta } $. 
It remains to estimate $\|u\|_{L_{2,2}(Q_{3/4})}$ in terms of $\| u_x \|_{L_{2,2} (Q_1)}$ above. To this end we introduce $\hat u$ in place of $u$ in \eqref{LEI_temp} and along the lines of our proofs of Caccioppoli inequalities \eqref{eq:cacc} or \eqref{eq:caccp} (this case is easier, since problem is linear and solutions are smooth). Next we repeat the above proof with $\hat u$.

\bibliography{literature}{}
\end{document}